\documentclass{article}

\usepackage{amsmath,amsthm,amsfonts}

\usepackage{graphics}
\usepackage{parskip}

\usepackage{fullpage}

\theoremstyle{plain}

\newtheorem{theorem}{Theorem}

\newtheorem{corollary}{Corollary}

\numberwithin{theorem}{section}
\numberwithin{proposition}{section}
\numberwithin{lemma}{section}
\numberwithin{corollary}{section}
\numberwithin{claim}{section}

\theoremstyle{definition}

\newtheorem{example}{Example}

\newtheorem{remark}{Remark}
\numberwithin{definition}{section}
\numberwithin{example}{section}
\numberwithin{question}{section}
\numberwithin{remark}{section}

\newcommand{\norm}[1]{\left\Vert #1\right\Vert}

\newcommand{\e}{\varepsilon}

\newcommand{\al}{\alpha}

\newcommand{\ii}{\infty}
\newcommand \beq{\begin{eqnarray*}}
\newcommand \eeq{\end{eqnarray*}}

\title{Fixed point results for a new mapping related to mean nonexpansive mappings}
\author{Torrey M. Gallagher}
\date{\today}

\begin{document}
\maketitle

\abstract{Mean nonexpansive mappings were first introduced in 2007 by Goebel and Japon Pineda and advances have been made by several authors toward understanding their fixed point properties in various contexts.  For any given $(\al_1, \al_2)$-nonexpansive mapping $T$ of a Banach space, many of the positive results have been derived from properties of the mapping $T_\al = \al_1 T + \al_2T^2= (\al_1I + \al_2T)\circ T$ which is nonexpansive.  However, the related mapping $T \circ (\al_1I + \al_2T)$ has not yet been studied.  In this paper, we investigate some fixed point properties of this new mapping and discuss relationships between $(\al_1I + \al_2T)\circ T$ and $T\circ(\al_1I + \al_2T)$.}

\section{Introduction}
Let $(X, \norm{\cdot})$ be a Banach space, and $C$ a nonempty subset of $X$.  A function $T : C\to X$ is called \textit{nonexpansive} if
\[
\norm{Tx-Ty} \leq \norm{x-y} \mbox{, for all } x,y \in C.
\]
It is a well-known application of Banach's Contraction Mapping Principle that every nonexpansive mapping $T : C \to C$ (where $C$ is closed, bounded, and convex) has an \textit{approximate fixed point sequence} $(x_n)_n$ in $C$.  That is, $(x_n)_n$ is a sequence for which $\norm{Tx_n - x_n} \to 0$.  The question of when nonexpansive maps have fixed points is much more difficult, however.  We say a Banach space $(X,\norm{\cdot})$ has \textit{the fixed point property for nonexpansive maps} if, for every closed, bounded, convex subset $C\neq \emptyset$ of $X$, every nonexpansive map $T : C\to C$ has a fixed point (that is, a point $x \in C$ for which $Tx=x$).   For a thorough introduction and survey of the history and results of metric fixed point theory, see \cite{topics,handbook}.

In this paper, we will be discussing the fixed point properties of a new mapping associated with the class of so-called ``mean nonexpansive maps,'' which were introduced in 2007 by Goebel and Jap\'on Pineda \cite{gjp07}.  Recent research in this area has proven to be fruitful, and the interested reader should see \cite{piasecki13} for a nearly complete survey of known results.


Specifically, we show in Theorems \ref{p afps} and \ref{fp} that this new mapping must have an approximate fixed point sequence and, in certain contexts, fixed points.  From this, we obtain a new proof of Theorem \ref{gjp}, which is originally due to Goebel and Jap\'on Pineda and provides a sufficient condition to guarantee the existence of a fixed point for a mean nonexpansive mapping.  We conclude with an open question about the new mapping defined herein.


\section{Preliminaries}
A function $T : C \to C$ is called \textit{mean nonexpansive} (or \textit{$\al$-nonexpansive}) if, for some $\al = (\al_1,\al_2, \ldots, \al_n)$ with $\sum_{k=1}^n \al_k = 1$, $\al_k \geq 0$ for all $k$, and $\al_1, \al_n >0$, we have
\[
\sum_{k=1}^n \al_k \norm{T^kx - T^ky} \leq \norm{x-y} \mbox{, for all } x,y \in C.
\]
It is clear that all nonexpansive mappings are mean nonexpansive, but the converse is not true.  That is, there exist mean nonexpansive mappings for which no iterate is nonexpansive (see Examples \ref{example1} and \ref{example2}).

Goebel and Jap\'on Pineda further suggested the class of $(\al, p)$-nonexpansive maps.  A function $T : C \to C$ is called \textit{$(\al, p)$-nonexpansive} if, for some $\al = (\al_1,\al_2, \ldots, \al_n)$ with $\sum_{k=1}^n \al_k = 1$, $\al_k \geq 0$ for all $k$, $\al_1, \al_n >0$, and for some $p \in [1,\ii)$,
\[
\sum_{k=1}^n \al_k \norm{T^kx - T^ky}^p \leq \norm{x-y}^p \mbox{, for all } x,y \in C.
\]

For simplicity, we will generally discuss the case when $n=2$.  That is, $T : C\to C$ is \textit{$((\al_1,\al_2),p)$-nonexpansive} if
for some $p \in [1,\ii)$, we have
\[
\al_1 \norm{Tx-Ty}^p + \al_2 \norm{T^2x - T^2y}^p \leq \norm{x-y}^p \mbox{, for all } x,y \in C.
\]
When $p=1$, we will say $T$ is $(\al_1,\al_2)$-nonexpansive.  When the multi-index $\al$ is not specified, we say $T$ is \textit{mean nonexpansive}.

It is easy to check that every $(\al, p)$-nonexpansive map for $p>1$ is also $\al$-nonexpansive, but the converse does not hold; that is, there is a mapping which is $\al$-nonexpansive that is not $(\al,p)$-nonexpansive for any $p>1$ (see \cite{piasecki13} for details). It is also easy to see that, by the triangle inequality, the mapping $T_\al := \al_1T + \al_2T^2$ is nonexpansive if $T$ is $(\al_1,\al_2)$-nonexpansive.  As noted in \cite{gjp07}, however, the nonexpansiveness of $T_\al$ is significantly weaker than the nonexpansiveness of $T$.  For example, $T_\al$ being nonexpansive does not even guarantee continuity of $T$, let alone any positive fixed point results \cite[Examples 3.5 and 3.6]{piasecki13}.  When $T$ is mean nonexpansive, Goebel and Jap\'on Pineda (and later Piasecki \cite[Theorems 8.1 and 8.2]{piasecki13}) were able to use the nonexpansiveness of $T_\al$ to prove some intriguing results about $T$, as summarized in the following theorem.

\begin{theorem}[Goebel and Jap\'on Pineda, Piasecki]\label{gjp}
Suppose $(X,\norm{\cdot})$ is a Banach space and $C \subset X$ is closed, bounded, convex, and $T : C \to C$ is $((\al_1,\al_2),p)$-nonexpansive for some $p\geq 1$.  Then $T$ has an approximate fixed point sequence, provided that $\al_2^p \leq \al_1$ (note that for $p=1$, this inequality reduces to $\al_1\geq 1/2$).  Furthermore, if $(X, \norm{\cdot})$ has the fixed point property for nonexpansive maps, then $T$ has a fixed point if $\al_2^p \leq \al_1$.
\end{theorem}

Note that $T_\al = \al_1 T + \al_2 T^2$, and so we may write $T_\al = (\al_1I + \al_2T)\circ T$, where $I$ denotes the identity mapping.  In the following, we will study properties of a related mapping given by $\tau_\al := T\circ (\al_1I+ \al_2T)$.  To the present author's knowledge, this mapping has not been studied in the literature.  Clearly if $T$ is linear (or, more generally, affine), then $(\al_1I + \al_2T)\circ T = T \circ (\al_1I + \al_2T)$.  This is not true in general, as shown in Example \ref{example1}.

It should be noted that, just as with $T_\al$, nonexpansiveness of $\tau_\al$ is not enough to even guarantee continuity of $T$, as the following example demonstrates.

\begin{example}
Let $f : [0,1] \to [0,1]$ be given by
\[
f(x) := \begin{cases}
1 & x=0\\
0 & x \neq 0
\end{cases}.
\]
Clearly $f$ is discontinuous.  Let $\al=(\al_1,\al_2)$ be arbitrary such that $\al_1, \al_2 >0$ and $\al_1 + \al_2 = 1$.  Then
\[
\al_1x +  \al_2 f(x) = \begin{cases}
\al_1x + \al_2 & x=0\\
\al_1x & x \neq 0
\end{cases}
=
\begin{cases}
\al_2 & x=0\\
\al_1x & x\neq 0
\end{cases}.
\]
But then $\al_1x + \al_2 f(x) \neq 0$ for any $x \in [0,1]$, and thus $f(\al_1x+\al_2f(x))=0$ for all $x \in [0,1]$ and $f \circ (\al_1I+\al_2f)$ is nonexpansive.
\end{example}

We now give an example, taken from \cite{goebelsims10} (see also \cite[Example 3.3]{piasecki13}), of a mean nonexpansive mapping defined on a closed, bounded, convex subset of a Banach space for which none of its iterates are nonexpansive.  This example will also demonstrate that $T_\al$ is generally not equal to $\tau_\al$.

\begin{example}\label{example1}
Let $(\ell^1, \norm{\cdot}_1)$ be the Banach space of absolutely summable sequences endowed with its usual norm.  Let $B_{\ell^1}$ denote the (closed) unit ball of $\ell^1$.  Then let $T : B_{\ell^1} \to B_{\ell^1}$ be given by
\[
T(x_1,x_2,x_3,\cdots) := \left(\tau(x_2), \, \frac{2}{3} x_3, \, x_4, \, \cdots\right),
\]
where $\tau : [-1,1] \to [-1,1]$ is given by
\[
\tau(t):=
\begin{cases}
2t+1 	& -1 \leq t \leq -1/2\\
0	& -1/2 \leq t \leq 1/2\\
2t-1	& 1/2 \leq t \leq 1.
\end{cases}
\]
It is easy to check (see \cite{piasecki13} for more details) that $T(B_{\ell^1}) \subseteq B_{\ell^1}$, $T$ is not $\norm{\cdot}_1$-nonexpansive, but $T$ is mean nonexpansive for $\al_1 = \al_2 = 1/2$.  That is, there exist two points $x, y \in B_{\ell^1}$ for which
\[
\norm{Tx-Ty}_1 > \norm{x-y}_1,
\]
but it is true that
\[
\frac{1}{2}\norm{Tx-Ty}_1 + \frac{1}{2} \norm{T^2x - T^2y}_1 \leq \norm{x-y}_1
\]
for all $x,y \in B_{\ell^1}$.

Then, for any $x$,
\[
\frac{1}{2}x + \frac{1}{2}Tx = \left( \frac{1}{2}x_1 + \frac{1}{2}\tau(x_2), \,\, \frac{1}{2}x_2 + \frac{1}{2}\left(\frac{2}{3}x_3\right), \,\, \frac{1}{2} x_3 + \frac{1}{2}x_4,\, \ldots\right)
\]
and
\[
T \circ \left( \frac{1}{2}I + \frac{1}{2}T\right) (x) = \left( \tau\left( \frac{1}{2}x_2 + \frac{1}{2}\left(\frac{2}{3}x_3\right) \right), \,\, \frac{2}{3}\left( \frac{1}{2} x_3 + \frac{1}{2}x_4 \right), \,\, \frac{1}{2} x_4 + \frac{1}{2}x_5, \, \ldots   \right).
\]
So, for instance, for $e_3 = (0,0,1,0,0,\ldots)$
\[
T\left(\frac{1}{2}e_3 + \frac{1}{2}Te_3\right) = \left( \tau\left( \frac{1}{3} \right), \, \frac{1}{3}, 0, 0, \ldots\right) = \left( 0, \frac{1}{3}, 0,0,\ldots\right)
\]
while
\begin{align*}
T_\al e_3 &= \frac{1}{2}Te_3 + \frac{1}{2} T^2e_3 \\
	&= \frac{1}{2}\left( 0, \frac{2}{3}, 0,0,\ldots \right) + \frac{1}{2}\left( \tau\left(\frac{2}{3}\right), 0, 0, 0, \ldots \right)\\
	&= \left( \frac{1}{6} , \, \frac{1}{3}, 0,0, \ldots \right).
\end{align*}

Thus, $T \circ \left(\frac{1}{2}I + \frac{1}{2} T\right) \neq \left(\frac{1}{2}I + \frac{1}{2}T\right) \circ T$.
\end{example}

Since we will also be discussing $(\al,p)$-nonexpansive maps, we give an example here to demonstrate that this class of mappings is also nontrivial.  Adapting the above example slightly gives us a nontrivial example of an $(\al, p)$-nonexpansive mapping on a closed, bounded, convex subset of a Banach space.  In particular, we have an example  of a $((1/2, 1/2), 2)$-nonexpansive mapping defined on the unit ball of $\ell^2$, the Hilbert space of square-summable sequences. 

\begin{example}\label{example2}
Let $(\ell^2,\norm{\cdot}_2)$ be the Banach space of square-summable sequences endowed with its usual norm.  Let $B_{\ell^2}$ denote the (closed) unit ball of $\ell^2$.  Then let $S : B_{\ell^2} \to B_{\ell^2}$ be given by
\[
S(x_1,x_2,x_3,\cdots) := \left(\sigma(x_2), \, \sqrt{\frac{2}{3}}\,\, x_3, \, x_4, \, \cdots\right),
\]
where $\sigma : [-1,1] \to [-1,1]$ is given by
\[
\sigma(t):=
\begin{cases}
\sqrt{2}t+(\sqrt{2}-1) 	& -1 \leq t \leq -t_0\\
0				& -t_0 \leq t \leq t_0\\
\sqrt{2}t-(\sqrt{2}-1)	& t_0 \leq t \leq 1,
\end{cases}
\]
where $t_0 = (\sqrt{2}-1)/\sqrt{2}$.

It is easy to check that $S(B_{\ell^2}) \subseteq B_{\ell^2}$, $S$ is not $\norm{\cdot}_2$-nonexpansive, but $S$ is mean nonexpansive for $\al_1 = \al_2 = 1/2$ and $p=2$.  That is, there exist two points $x, y \in B_{\ell^2}$ for which
\[
\norm{Sx-Sy}_2 > \norm{x-y}_2,
\]
but it is true that
\[
\frac{1}{2}\norm{Sx-Sy}_2^2 + \frac{1}{2} \norm{S^2x - S^2y}_2^2 \leq \norm{x-y}_2^2
\]
for all $x,y \in B_{\ell^2}$.
\end{example}

Finally, the present author showed in \cite{demiclosed} that, given $\al = (\al_1,\ldots,\al_n)$ and an $(\al, p)$-nonexpansive map $T: C \to C$, the mapping $\widetilde T : C^n \to C^n$ is nonexpansive when restricted to the diagonal of $C^n$ (i.e. $D := \{ (x,x,\ldots,x) : x \in C \}$) when $X^n$ is equipped with the norm
\[
\norm{(x_1,x_2, \ldots, x_n)}_{\al,p} := \left( \al_1 \norm{x_1}^p + \al_2 \norm{x_2}^p + \cdots + \al_n \norm{x_n}^p \right)^{\frac{1}{p}}.
\]
That is, for all $x_1, y_1, x_2, y_2 \ldots, x_n, y_n \in C$, we have
\[
\norm{\widetilde T (x_1, x_2, \ldots, x_n) - \widetilde T(y_1,y_2,\ldots, y_n)}_{\al, p} \leq \norm{(x_1, x_2, \ldots, x_n) - (y_1, y_2, \ldots, y_n)}_{\al, p}.
\]
This observation was used in particular to establish the so-called ``demiclosedness principle'' for mean nonexpansive mappings defined on uniformly convex spaces or spaces satisfying Opial's property.  We use the nonexpansiveness of $\widetilde T$ in the proofs that follow.


\section{Results for $\al = (\al_1, \al_2)$}
\begin{theorem}\label{afps}
Suppose $(X,\norm{\cdot})$ is a Banach space, $C \subset X$ is closed, bounded, convex, and $T : C\to C$ is $(\al_1,\al_2)$-nonexpansive.  Then there exist sequences $(x_n)_n$ and $(y_n)_n$ in $C$ for which
\begin{equation}\begin{cases}
\norm{T(\al_1 x_n + \al_2 y_n) - x_n} \to_n 0, \quad \text{ and}\\
\norm{T^2(\al_1 x_n + \al_2 y_n) - y_n} \to_n 0.
\end{cases}
\end{equation}

In particular, we can deduce
\begin{equation}\begin{cases}
\norm{Tx_n - y_n} \to_n 0, \quad \text{ and}\\
\norm{T(\al_1 x_n + \al_2 Tx_n) - x_n} \to_n 0.
\end{cases}
\end{equation}

In other words, $(x_n)_n$ is an approximate fixed point sequence for $\tau_\al:= T \circ (\al_1I + \al_2 T)$.
\end{theorem}
\begin{proof}
Consider the space $(X^2, \norm{\cdot}_\al)$, where 
\[
\norm{(x,y)}_\al := \al_1 \norm{x} + \al_2 \norm{y},
\]
and the mapping $\widetilde T : C^2 \to C^2$ given by
\[
\widetilde T (x,y) := (Tx, \,\,T^2y).
\]
Let $D := \{ (x,x) : x \in C \} \subset C$.  Then, using the fact that $T$ is $(\al_1,\al_2)$-nonexpansive; i.e.
\[
\al_1 \norm{Tx-Ty} + \al_2 \norm{T^2x - T^2y} \leq \norm{x-y}
\]
for all $x,y \in C$, we see that $\widetilde T \big|_D : D \to C^2$ is nonexpansive:
\begin{align*}
\norm{\widetilde T (x,x) - \widetilde T(y,y)}_\al &= \norm{ (Tx - Ty,\,\, T^2x - T^2y)}_\al\\
	&= \al_1 \norm{Tx-Ty} + \al_2 \norm{T^2x - T^2y}\\
	&\leq \norm{x-y}\\
	&= \al_1 \norm{x-y} + \al_2 \norm{x-y}\\
	&= \norm{(x,x)-(y,y)}_\al.
\end{align*}
However, $\widetilde T\big|_D$ is not a self-mapping of $D$, so much of the usual theory for nonexpansive mappings does not apply.  In light of this, define a new mapping $J : C^2 \to C^2$ by
\begin{align*}
J(x,y) :&= \widetilde T ( \al_1x + \al_2 y,\,\, \al_1 x + \al_2 y)\\
	&= (T(\al_1 x + \al_2 y),\,\, T^2 (\al_1 x + \al_2 y)).
\end{align*}

Then $J$ is nonexpansive on $C^2$.  Indeed, for any $(x,y)$ and $(u,v)$ in $C^2$, we know that $(\al_1x+\al_2y,\,\, \al_1x + \al_2y)$, $(\al_1u+\al_2v,\,\, \al_1u + \al_2v) \in D$.  Thus, since $\widetilde T\big|_D$ is nonexpansive, we have
\begin{align*}
\norm{J(x,y) - J(u,v)}_\al &= \norm{ \widetilde T ( \al_1x + \al_2 y, \,\,\al_1 x + \al_2 y) -  \widetilde T ( \al_1u + \al_2 v,\,\, \al_1 u + \al_2 v)}_\al\\
	&\leq \norm{ ( \al_1x + \al_2 y, \,\,\al_1 x + \al_2 y) - ( \al_1u + \al_2 v,\,\, \al_1 u + \al_2 v)}_\al\\
	&= \norm{ (\al_1 (x-u) + \al_2 (y-v),\,\, \al_1 (x-u) + \al_2 (y-v))}_\al\\
	&= \al_1 \norm{\al_1 (x-u) + \al_2 (y-v)} + \al_2\norm{\al_1 (x-u) + \al_2 (y-v)}\\
	&= \norm{\al_1 (x-u) + \al_2 (y-v)}\\
	&\leq \al_1 \norm{x-u} + \al_2 \norm{y-v}\\
	&= \norm{(x,y)-(u,v)}_\al.
\end{align*}
Since $C$ is closed, bounded, and convex in $X$, it is easy to see that $C^2$ is closed, bounded, and convex in $X^2$.  Thus, since $J : C^2 \to C^2$ is nonexpansive, we know that it must admit an approximate fixed point sequence $(x_n, y_n)_n$.  That is, a sequence for which
\[
\norm{J(x_n,y_n) - (x_n,y_n)}_\al \to_n 0.
\]
Examining the last line more closely, we see that
\[
\norm{J(x_n,y_n) - (x_n,y_n)}_\al \to_n 0 \iff \begin{cases} \norm{T(\al_1 x_n + \al_2 y_n) - x_n} \to_n 0, \quad \text{ and}\\  \norm{T^2(\al_1 x_n + \al_2 y_n) - y_n} \to_n 0.\end{cases}
\]
This completes the proof of (1) in the statement of the theorem.

Now let us prove (2).  Note that all mean nonexpansive mappings are Lipschitzian with $k(T) \leq \al_1^{-1}$, where $k(T)$ denotes the Lipschitz constant of $T$ (indeed, more is true: all mean nonexpansive maps are \textit{uniformly Lipschitzian} \cite[Chapter 4]{piasecki13}).  Then, by (1),
\[
\norm{T^2(\al_1x_n + \al_2y_n) - Tx_n} \leq k(T)\norm{T(\al_1x_n + \al_2y_n) - x_n} \to_n 0.
\]
Then for all $n$ we have
\[
\norm{y_n - Tx_n} - \norm{T^2(\al_1x_n + \al_2y_n) - y_n} \leq \norm{T^2(\al_1x_n + \al_2y_n) - Tx_n},
\]
which then yields
\begin{align*}
0\leq \limsup_n \norm{y_n - Tx_n} &= \limsup_n \left(\norm{y_n - Tx_n} - \norm{T^2(\al_1x_n + \al_2y_n) - y_n}\right) \\
	&\leq \limsup_n \norm{T^2(\al_1x_n + \al_2y_n) - Tx_n}=0,
\end{align*}
Hence, $\norm{Tx_n - y_n} \to_n 0$.  To complete the proof of (2), we only need to use the fact that $T$ is Lipschitz (in fact, $T$ only needs to be continuous for this argument to work).  For simplicity, let $z_n := \al_1x_n + \al_2Tx_n$ and note that
\begin{align*}
\norm{Tz_n - x_n} &\leq \norm{Tz_n - T(\al_1x_n + \al_2 y_n)} + \norm{T(\al_1x_n + \al_2 y_n) - x_n}\\
	&\leq k(T) \norm{z_n - (\al_1x_n + \al_2 y_n)} + \norm{T(\al_1x_n + \al_2 y_n) - x_n}\\
	&= k(T) \al_2 \norm{Tx_n -y_n} + \norm{T(\al_1x_n + \al_2 y_n) - x_n}\\
	&\to_n 0.
\end{align*}
Hence, $(x_n)_n$ is an approximate fixed point sequence for $\tau_\al$, and the proof of the theorem is complete.
\end{proof}

The above theorem holds in more generality.  In particular, the same result holds for $(\al, p)$-nonexpansive maps as summarized in the following theorem.

\begin{theorem}\label{p afps}
Suppose $(X, \norm{\cdot})$ is a Banach space, $C \subset X$ is closed, bounded, convex, and $T : C \to C$ is $((\al_1,\al_2), p)$-nonexpansive for some $p\geq 1$.  Then there exist sequences $(x_n)_n$ and $(y_n)_n$ satisfying (1) and (2) from Theorem \ref{afps}.  In particular, the sequence $(x_n)_n$ is an approximate fixed point sequence for $\tau_\al$.
\end{theorem}

The proof of Theorem \ref{p afps} is entirely similar to the proof of Theorem \ref{afps}, so we present only the portions which differ.

\begin{proof}
Let $(X,\norm{\cdot})$ be a Banach space, $C \subset X$ closed, bounded, convex, and $T : C \to C$ $(\al,p)$-nonexpansive for some $p\geq 1$.  Consider the space $(X^2, \norm{\cdot}_{\al,p})$, where
\[
\norm{(x,y)}_{\al,p} := \left( \al_1 \norm{x}^p + \al_2 \norm{y}^p \right)^{\frac{1}{p}}.
\]
Now define the functions $\widetilde T$ and $J$ just as in the proof of Theorem \ref{afps}, and notice that $J$ is nonexpansive on $C^2$ since $t \mapsto t^p$ is a convex function for $p\geq 1$.  Indeed,
\begin{align*}
\norm{J(x,y) - J(u,v)}_{\al,p}^p &= \norm{\widetilde T(\al_1x+\al_2y, \al_1x+\al_2y) - \widetilde T(\al_1u + \al_2 v, \al_1 u + \al_2 v)}_{\al,p}^p\\
	&\leq \norm{(\al_1 (x-u) + \al_2(y-v), \al_1(x-u)+\al_2(y-v))}_{\al,p}^p\\
	&= \norm{\al_1 (x-u) + \al_2 (y-v)}^p\\
	&\leq \left( \al_1 \norm{x-u} + \al_2 \norm{y-v} \right)^p\\
	&\leq \al_1 \norm{x-u}^p + \al_2 \norm{y-v}^p\\
	&= \norm{(x,y) - (u,v)}_{\al,p}^p.
\end{align*}
The remainder of the proof follows as above.
\end{proof}

We have a corollary regarding the form of an approximate fixed point sequence for the mapping $T_\al$.

\begin{corollary}\label{T alpha afps}
Suppose $C$ is closed, bounded, convex and $T : C \to C$ is $((\al_1,\al_2),p)$-nonexpansive.  Then $T_\al$ admits an approximate fixed point sequence $(z_n)_n$ of the form
\[
z_n = \al_1 x_n + \al_2 Tx_n
\]
where $(x_n)_n$ is the approximate fixed point sequence for $\tau_\al$ from Theorem \ref{p afps}.
\end{corollary}
\begin{proof}
From Theorem \ref{p afps}, we know there is a sequence $(x_n)_n$ satisfying
\[
\norm{T(\al_1x_n + \al_2Tx_n) - x_n}\to_n 0
\]
and, since $T$ is Lipschitz, it is easy to see that this implies
\[
\norm{T^2(\al_1x_n + \al_2Tx_n) - Tx_n}\to_n 0.
\]
Let $z_n = \al_1x_n + \al_2Tx_n$.  Then
\begin{align*}
\norm{T_\al z_n - z_n} &= \norm{\al_1 (Tz_n - x_n) + \al_2(T^2z_n - Tx_n)}\\
	&\leq \al_1 \norm{Tz_n - x_n} + \al_2 \norm{T^2z_n - x_n}\\
	&\to_n 0.
\end{align*}
Thus, $(z_n)_n$ is an approximate fixed point sequence for $T_\al$.
\end{proof}

The theorems above tell us a bit more when the set $C^2$ has the fixed point property for nonexpansive maps.  This occurs whenever, for example, $(X, \norm{\cdot})$ is uniformly convex since $(X^2, \norm{\cdot}_{\al,p})$ is also uniformly convex (when $p>1$) by a theorem of Clarkson \cite{clarkson36}.  The proof of the following corollary follows immediately from the proof of Theorem \ref{p afps} when ``approximate fixed point sequence'' is replaced by ``fixed point.''

\begin{corollary}\label{fp}
Suppose $C \subset X$ is closed, bounded, and convex is such that $(C^2, \norm{\cdot}_{\al,p})$ has the fixed point property for nonexpansive maps.  If $T : C \to C$ is $((\al_1,\al_2),p)$-nonexpansive, then there exist points $x, y \in C$ for which
\[
\begin{cases}
T(\al_1x+\al_2y) = x, \quad \text{ and}\\
T^2 (\al_1x+\al_2y) = y.
\end{cases}
\]
In particular, we deduce that
\[
\begin{cases}
Tx = y, \quad \text{ and}\\
T(\al_1x+\al_2Tx) = x.
\end{cases}
\]
That is, $\tau_\al$ has a fixed point.
\end{corollary}

This leads us immediately to the analogue of Corollary \ref{T alpha afps} above.

\begin{corollary} \label{T alpha fp}
Suppose $(X,\norm{\cdot})$ has the fixed point property for nonexpansive maps, $C \subset X$ is closed, bounded, and convex, and $T : C \to C$ is $(\al_1,\al_2)$-nonexpansive.  Then $T_\al = \al_1 T + \al_2 T^2$ has at least one fixed point $y$ of the form
\[
y = \al_1 x + \al_2 Tx
\]
for some $x \in C$.
\end{corollary}

The approximate fixed point sequence $(x_n)_n$ for $\tau_\al$ yields a new proof of Theorem \ref{gjp}.
\begin{corollary}
Suppose $C$ is closed, bounded, convex and $T : C \to C$ is $(\al,p)$-nonexpansive for $p \geq 1$.  Then $T$ has an approximate fixed point sequence, provided that $\al_2^p \leq \al_1$.  If $(C^2, \norm{\cdot}_{\al,p})$ has the fixed point property for nonexpansive maps, then $T$ has a fixed point if $\al_2^p \leq \al_1$.
\end{corollary}
\begin{proof}
Suppose $T : C\to C$ is $(\al,p)$-nonexpansive for some $p\geq 1$ and $\al_2 \leq \al_1^{\frac{1}{p}}$, and let $(x_n)_n$ be the approximate fixed point sequence for $\tau_\al$ given by Theorem \ref{p afps}.  Fix $\e>0$.  Since $\norm{\tau_\al x_n-x_n} \to_n 0$, we can find $n$ large enough that
\[
\norm{\tau_\al x_n - x_n} \leq \left( \frac{\al_1^{\frac{1}{p}} - \al_2}{\al_1^{\frac{1}{p}}} \right) \e.
\]
For simplicity, let $c:= \left(\al_1^{\frac{1}{p}}-\al_2\right)\al_1^{-\frac{1}{p}}=1-\al_2\al_1^{-\frac{1}{p}}$.  Then we have
\begin{align*}
\norm{Tx_n - x_n} &= \norm{Tx_n -\tau_\al x_n + \tau_\al x_n - x_n}\\
	&\leq \norm{Tx_n - T(\al_1x_n + \al_2Tx_n)} + \norm{T(\al_1x_n + \al_2Tx_n) - x_n}\\
	&\leq \al_1^{-\frac{1}{p}} \norm{x_n - \al_1x_n - \al_2Tx_n} + c\e\\
	&= \al_2 \al_1^{-\frac{1}{p}} \norm{Tx_n - x_n} + c\e.
\end{align*}
Thus, $\left(1-\al_2\al_1^{-\frac{1}{p}}\right)\norm{Tx_n - x_n} \leq c\e \iff \norm{Tx_n-x_n} \leq \e$.  Thus, $(x_n)_n$ is an approximate fixed point sequence for $T$.

In the case when $(C^2, \norm{\cdot}_{\al,p})$ has the fixed point property, taking $\e=0$ in the above argument yields the desired result.  That is, $Tx=x$, where $T$ is the fixed point of $\tau_\al$ which is guaranteed to exist by Corollary \ref{fp}.
\end{proof}


\section{Results for arbitrary $\al$}
Results very similar to the ones above hold for $(\al, p)$-nonexpansive mappings with $\al$ of arbitrary length, and the proofs are nearly identical.  We state them here for completeness as well as providing pertinent details for adapting the proofs.

\begin{theorem}\label{n afps}
Suppose $(X, \norm{\cdot})$ is a Banach space, $C \subset X$ is closed, bounded, convex, and $T : C \to C$ is $(\al,p)$-nonexpansive for $p\geq 1$ and some $\al=(\al_1,\ldots,\al_n)$.  Without loss of generality, let us assume that each $\al_j>0$ (see the remark which follows the proof for more details). Then there exist sequences $(x_m^j)_m$, $j=1,2,\ldots,n$, in $C$ for which
\[
\begin{cases}
\norm{ T(\al_1x_m^{(1)} + \al_2 x_m^{(2)} + \cdots + \al_n  x_m^{(n)}) - x_m^{(1)}} \to 0\\
\norm{ T^2(\al_1x_m^{(1)} + \al_2 x_m^{(2)} + \cdots + \al_n  x_m^{(n)}) - x_m^{(2)}} \to 0\\
\quad\quad\quad\quad\quad\quad\quad\quad \vdots\\
\norm{ T^n(\al_1x_m^{(1)} + \al_2 x_m^{(2)} + \cdots + \al_n  x_m^{(n)}) - x_m^{(n)}} \to 0
\end{cases}
\]
In particular, we can deduce that
\[
\norm{T(\al_1 x_m^{(1)} + \al_2 Tx_m^{(1)} + \cdots + \al_n T^{n-1} x_m^{(1)}) - x_m^{(1)}} \to 0.
\]
That is, $\tau_\al := T \circ (\al_1 I + \al_2T + \cdots + \al_n T^{n-1})$ has an approximate fixed point sequence.
\end{theorem}
\begin{proof}
Just as in the proof of Theorem \ref{afps}, we define $\widetilde T,\, J : C^n \to C^n$ via
\[
\widetilde T (x_1, x_2, \ldots, x_n) := (Tx_1, T^2x_2, \ldots, T^nx_n), \quad J(x_1,x_2,\ldots,x_n):= \widetilde T (\overline{x},\,\overline{x}\,,\ldots,\, \overline{x}),
\]
where $\overline{x}:= \al_1x_1 + \al_2x_2+ \cdots + \al_nx_n$.  Just as before, $\widetilde T\big|_D$ is nonexpansive in the norm
\[
\norm{(x_1,x_2,\ldots,x_n)}_{\al,p} := \left( \al_1 \norm{x_1}^p + \al_2 \norm{x_2}^p + \cdots + \al_n \norm{x_n}^p \right)^{\frac{1}{p}}
\]
and $D := \{ (x,x,\ldots,x): x \in C\}$.  Since $\overline{x} \in D$, $J$ is also nonexpansive, and since $C^n$ is closed, bounded, and convex, we know that $J$ must have an approximate fixed point sequence, which we will denote $((x^{(1)}_m, \,x^{(2)}_m,\ldots,\,x^{(n)}_m))_m$ in $C^n$.  This establishes the first part of the theorem.

To prove that $\tau_\al$ has an approximate fixed point sequence, we will first denote 
\[
\overline{x}_m:=\al_1x^{(1)}_m + \al_2x_m^{(2)} + \cdots + \al_n x_m^{(n)}
\]
and note that
\begin{align*}
\norm{Tx^{(1)}_m - x^{(2)}_m} &\leq \norm{Tx^{(1)}_m - T^2\overline{x}_m} + \norm{T^2 \overline{x}_m - x^{(2)}_m}\\
	&\leq k(T)\norm{x_m^{(1)} - T \overline{x}_m}+\norm{T^2 \overline{x}_m - x^{(2)}_m}\\
	&\to 0
\end{align*}
as $m \to \ii$.  Thus, $\norm{Tx^{(1)}_m - x^{(2)}_m} \to 0$.  Entirely similarly,
\[
\norm{T^j x_m^{(1)} - x_m^{(j+1)}} \to 0
\]
for $j=1,\ldots,n-1$.  From this, let 
\[
z_m := \al_1x_m^{(1)} + \al_2 T x_m^{(1)} + \cdots + \al_n T^{n-1} x_m^{(1)}
\] 
(so that $\tau_\al x_m^{(1)} = Tz_m$) and observe that
\begin{align*}
\norm{\tau_\al x_m^{(1)} - x_m^{(1)}} &\leq \norm{\tau_\al x_m^{(1)} - T \overline{x}_m} + \norm{T \overline{x}_m - x_m^{(1)}}\\
	&\leq k(T) \norm{z_m - \overline{x}_m} + \norm{T \overline{x}_m - x_m^{(1)}}\\
	&\leq k(T) \left( \sum_{j=2}^n \al_j \norm{T^{j-1} x_m^{(1)} - x_m^{(j)}}  \right) +\norm{T \overline{x}_m - x_m^{(1)}}\\
	&\to 0
\end{align*}
as $m \to \ii$.  Thus, $(x_m^{(1)})_m$ is an approximate fixed point sequence for $\tau_\al$, and the proof is complete.
\end{proof}

\begin{remark}
In the event that some of the $\al_j$'s are equal to 0, the only problem that arises in the above proof is that
\[
\norm{(x_1,x_2,\ldots,x_n)}_{\al,p} = \left(\al_1 \norm{x_1}^p + \al_2 \norm{x_2}^p + \cdots + \al_n \norm{x_n}^p \right)^{\frac{1}{p}}
\]
no longer defines a norm on $X^n$.  To get around this, let $\{k_1, k_2, \ldots, k_\nu\} = \{ \al_j : \al_j \neq 0\}$, where $1=k_1 < k_2 < \ldots < k_\nu = n$ (recall that $\al_1, \al_n >0$ by definition).  Then
\begin{align*}
\norm{(x_1,x_2,\ldots,x_\nu)}_{\al,p} :&= \left(\al_{k_1} \norm{x_1}^p + \al_{k_2} \norm{x_2}^p + \cdots + \al_{k_\nu} \norm{x_n}^p \right)^{\frac{1}{p}}\\
	&= \left(\al_{1} \norm{x_1}^p + \al_{k_2} \norm{x_2}^p + \cdots + \al_{n} \norm{x_n}^p \right)^{\frac{1}{p}}
\end{align*}
does indeed define a norm on $X^\nu$, and 
\begin{align*}
\widetilde T (x_1,x_2,\ldots,x_\nu) :&= (T^{k_1} x_1,\, T^{k_2} x_2, \ldots,\, T^{k_\nu} x_\nu)\\
	&= (T x_1,\, T^{k_2} x_2, \ldots,\, T^n x_\nu)
\end{align*}
is nonexpansive when restricted to the diagonal of $C^\nu$.
\end{remark}

Just as in the previous section, the same proof above can be adapted to show the following.

\begin{theorem} \label{n fp}
Suppose $C \subset X$ is closed, bounded, and convex is such that $(C^n, \norm{\cdot}_{\al,p})$ has the fixed point property for nonexpansive maps.  Then if $T : C \to C$ is $(\al, p)$-nonexpansive for some $p\geq 1$ and $\al=(\al_1,\ldots,\al_n)$ (without loss of generality, each $\al_j >0$), then there exist $x_1,x_2,\ldots,x_n \in C$ for which
\[
\begin{cases}
T \overline{x} = x_1\\
T^2 \overline{x} = x_2\\
\quad\quad \vdots\\
T^n \overline{x} = x_n
\end{cases}
\]
where $\overline{x} := \al_1 x_1 + \al_2 x_2 + \cdots + \al_n x_n$.  In particular, $\tau_\al x_1 = x_1$.
\end{theorem}

Furthermore, analogues of Corollaries \ref{T alpha afps} and \ref{T alpha fp} are readily available.

\begin{corollary}
Suppose $C \subset X$ is closed, bounded, convex, and $T : C\to C$ is $(\al,p)$-nonexpansive for some $\al=(\al_1,\ldots,\al_n)$ and $p \geq 1$.  Then $T_\al$ admits an approximate fixed point sequence $(z_n)_n$ of the form
\[
z_n := \al_1 x_n + \al_2 Tx_n + \cdots + \al_n T^{n-1}x_n,
\]
where $(x_n)_n$ is the approximate fixed point sequence for $\tau_\al$ guaranteed by Theorem \ref{n afps}.  Further, if we suppose that $(C^n, \norm{\cdot}_{\al,p})$ has the fixed point property for nonexpansive maps, then $T_\al$ has a fixed point $z$ of the form
\[
z := \al_1 x + \al_2 Tx + \cdots + \al_n T^{n-1}x,
\]
where $x$ is the fixed point of $\tau_\al$ guaranteed by Theorem \ref{n fp}.
\end{corollary}

Goebel and Jap\'on Pineda proved a version of Theorem \ref{gjp} for mean nonexpansive mappings with arbitrary length multi-index, which was again improved later by Piasecki.

\begin{theorem}[Goebel and Jap\'on Pineda, Piasecki]\label{n gjp}
If $C \subset X$ is closed, bounded, convex, and $T : C \to C$ is $(\al, p)$-nonexpansive for some $p \geq 1$ and $\al=(\al_1,\ldots,\al_n)$, then $T$ has an approximate fixed point sequence, provided that 
\[
(1-\al_1)\left( 1- \al_1^{\frac{n-1}{p}}\right) \leq \al_1^{\frac{n-1}{p}} \left( 1-\al_1^{\frac{1}{p}} \right).
\] 
Further, if $X$ has the fixed point property for nonexpansive maps, then $T$ has a fixed point if the above inequality holds.
\end{theorem}

While our techniques gave an alternate proof of this theorem in the case when $n=2$, it is not immediately clear that our techniques will give an alternative proof of Theorem \ref{n gjp}.  The techniques that we have used so far yield the result of Theorem \ref{n gjp} in  the case when $\al = (\al_1,\al_2,\al_3)$ and $p=1$.  Specifically, we can prove that an $(\al_1,\al_2,\al_3)$ nonexpansive map $T$ will have an approximate fixed point sequence if
\[
1-2\al_1^2 \leq \al_2.
\]
This inequality shows that for $\al_1 \geq \sqrt{2}/2$, any choice of $\al_2$ and $\al_3$ is valid.  Now for the argument.  Let $(x_n)_n$ be the approximate fixed point sequence for $\tau_\al$ guaranteed by Theorem \ref{n afps}.  Then
\[
\norm{Tx_n - x_n} \leq \norm{Tx_n - \tau_\al x_n} + \norm{\tau_\al x_n - x_n}.
\]
Choose $n$ large enough so that $\norm{\tau_\al x_n - x_n} < \e$.  Then we have
\begin{align*}
\norm{Tx_n - x_n} &\leq k(T) \norm{x_n - (\al_1x_n + \al_2 Tx_n + \al_3 T^2 x_n)} + \e\\
	&= k(T) \norm{\al_2(x_n - Tx_n) + \al_3(x_n - T^2x_n)} + \e\\
	&\leq k(T) (\al_2 \norm{x_n - Tx_n} + \al_3 \norm{x_n - T^2 x_n})+\e\\
	&\leq k(T) (\al_2 \norm{x_n - Tx_n} + \al_3 (\norm{x_n - Tx_n} + k(T)\norm{x_n - Tx_n}))+\e\\
	&\leq \al_1^{-1}(\al_2 \norm{x_n - Tx_n} + \al_3 (\norm{x_n - Tx_n} + \al_1^{-1}\norm{x_n - Tx_n}))+\e
\end{align*}
Thus,
\[
(\al_1^2-\al_1(\al_2+\al_3)-\al_3)\norm{x_n - Tx_n} \leq c\e,
\]
Where $c$ is a positive constant.  This inequality is only meaningful if $\al_1^2 - \al_1(\al_2+\al_3) - \al_3 \geq 0$.  Rewriting $\al_2 + \al_3 = 1-\al_1$ and $\al_1 + \al_3 = 1-\al_2$ yields $1-2\al_1^2 \leq \al_2$, as desired.  Very similarly to Goebel and Jap\'on Pineda's methods involving $T_\al$, our method presented here relies almost entirely on the triangle inequality and there should be room for improvement.

It should also be noted that Goebel and Jap\'on Pineda were able to improve the lower bound on $\al_1$ in the case when $p=1$ and $\al=(\al_1,\al_2,\al_3)$.

\begin{theorem}[Goebel and Jap\'on Pineda]
Suppose $C \subset X$ is closed, bounded, convex and $T : C \to C$ is $(\al_1, \al_2, \al_3)$-nonexpansive with
\[
\al_1 \in \left[\left. \frac{1}{2}, \frac{\sqrt{2}}{2} \right)\right. \quad \text{ and } \quad \frac{1}{2}(1-\al_1) \leq \al_2.
\]
Then $T$ has an approximate fixed point sequence.
\end{theorem}

Our technique only yields $1-2\al_1^2 \leq \al_2$, which implicitly forces $\al_1 > 1/2$ (since $\al_3$ cannot be 0).  Furthermore, it is easy to check that
\[
\frac{1}{2} (1-\al_1) < 1- 2\al_1^2
\]
for $\al_1$ with $0 < \al_1 \leq \sqrt{2}/2$.  That is, our lower bound for $\al_2$ is worse than Goebel and Jap\'on Pineda's in this special case.  Finally, we will state the estimates that our technique yields in the general case.  It is not immediately clear (even in the case when $n=4$) that our estimates are even at least as good as those of Piasecki and Goebel and Jap\'on stated in Theorem \ref{n gjp}, so we will state them as a remark.

\begin{remark}
Suppose $T : C \to C$ is $(\al,p)$-nonexpansive for some $\al=(\al_1,\ldots,\al_n)$ and $p\geq 1$.  Then $T$ has an approximate fixed point sequence if
\begin{align*}
1 &\geq k(T) \left( \left(\sum_{j=2}^{n} \al_j\right) + \left(\sum_{j=3}^{n} \al_j\right)k(T) + \left(\sum_{j=4}^{n} \al_j\right)k(T^2) + \cdots + \al_nk(T^{n-2})\right)\\
	&= k(T) \sum_{m=2}^n \left( \sum_{j=m}^n \al_j \right) k(T^{m-2}),
\end{align*}
where $k(T^j)$ is the Lipschitz constant of $T^j$ and $T^0 := I$.  To see this, fix $\e>0$ and let $x \in C$ be a point for which $\norm{\tau_\al x - x} \leq \e$.  Then
\begin{align*}
\norm{x-Tx} &\leq \norm{Tx-\tau_\al x} + \norm{\tau_\al x - x}\\
	&\leq k(T) \norm{x - (\al_1x + \al_2Tx + \cdots + \al_n T^{n-1}x)} + \e\\
	&\leq k(T) \sum_{j=1}^{n-1} \al_{j+1}\norm{x-T^j x}.
\end{align*}
Now observe that 
\[
\norm{x-T^jx} \leq \norm{x-T^{j-1}x} + \norm{T^{j-1}x-T^j x} \leq \norm{x-T^{j-1}x} + k(T^{j-1})\norm{x-Tx}.
\]
Iterating this estimate in the above yields
\begin{align*}
\norm{x-Tx} &\leq k(T) \sum_{j=1}^{n-1} \al_{j+1}\norm{x-T^j x} + \e\\
	& \leq k(T)\left( \sum_{j=2}^n \al_j + k(T)\sum_{j=3}^n \al_j + \cdots + k(T^{n-2})\al_n \right)\norm{x-Tx} + \e.
\end{align*}
Thus,
\[
\left(1 -  k(T) \sum_{m=2}^n \left( \sum_{j=m}^n \al_j \right) k(T^{m-2}) \right) \norm{x-Tx} \leq \e
\]
and hence $T$ will have an approximate fixed point sequence if 
\[
1 -  k(T) \sum_{m=2}^n \left( \sum_{j=m}^n \al_j \right) k(T^{m-2}) \geq 0.
\]
\end{remark}


\section{Questions}
There is a natural question underlying this entire study: is $\tau_\al$ necessarily nonexpansive?  If it is, then most of the above results would be greatly simplified, though less interesting.  We know that $T_\al$ is nonexpansive by a straightforward application of the triangle inequality.  However, a priori estimates for $\tau_\al$ do not have the same promise. If $T$ is nonexpansive, then it is easy to see that $\tau_\al$ is also nonexpansive, but if $T$ is assumed only to be mean nonexpansive, the routine estimate for its Lipschitz constant is less useful.  Indeed, one would naively find that
\[
k(\tau_\al) = k(T \circ (\al_1 I + \al_2 T)) \leq \al_1^{-1} (\al_1 + \al_2\al_1^{-1}) = 1 + \al_2\al_1^{-2},
\]
and $1+\al_2\al_1^{-2}$ clearly exceeds 1.  The two examples given above both have nonexpansive $\tau_\al$ despite the original mapping failing to be nonexpansive.

Finally, we reiterate the question that Goebel and Jap\'on Pineda originally posed: can anything, positive or negative, be said about $(\al_1,\al_2)$-nonexpansive mappings for which $\al_1<1/2$ (more generally, $\al_2^p \leq \al_1$ for $((\al_1,\al_2),p)$-nonexpansive mappings or $(1-\al_1)\left( 1- \al_1^{\frac{n-1}{p}}\right) \leq \al_1^{\frac{n-1}{p}} \left( 1-\al_1^{\frac{1}{p}} \right)$ for $(\al,p)$-nonexpansive mappings with $\al$ of length $n$ \cite{garciapiasecki12})?  Some partial results in special cases are known; for instance, if $T$ is $(\al_1,\al_2,\al_3)$-nonexpansive with $\al_1 \in [1/2, 1/\sqrt{2})$ and $\al_2 \geq (1-\al_1)/2$, then $T$ has an approximate fixed point sequence \cite{gjp07}.  A few other, more general, facts are also known.  For instance, Hilbert spaces have the fixed point property for $((\al_1,\al_2),2)$-nonexpansive mappings \cite[Corollary 3.7]{garciapiasecki12}, and all $(\al,p)$-nonexpansive mappings ($p>1$) defined on a closed, bounded, convex subset of a uniformly convex space are such that $I-T$ is demiclosed at 0 \cite{demiclosed}.  However, there are no theorems or counterexamples that treat small values of $\al_1$.

\end{document}